\documentclass[12pt]{amsart}
\usepackage{amssymb}
\newtheorem{theorem}{Theorem}
\newtheorem{proposition}[theorem]{Proposition}
\newtheorem{lemma}[theorem]{Lemma}
\newtheorem{example}[theorem]{Example}
\def\too{\to}
\def\C{\Bbb C}
\def\G{\Bbb G}
\def\D{\Bbb D}
\def\M{\mathcal M}
\def\N{\Bbb N}

\def\phi{\varphi}
\def\wdtl{\widetilde}
\def\CD{\mathcal D}
\def\O{\mathcal O}
\def\Om{\Omega}

\def\span{\operatorname{span}}
\def\trace{\operatorname{trace}}
\def\ds{\displaystyle}
\begin{document}

\title{Kobayashi-Royden pseudometric vs. Lempert function}

\author{Nikolai Nikolov}
\address{Institute of Mathematics and Informatics\\ Bulgarian Academy of
Sciences\\1113 Sofia, Bulgaria}\email{nik@math.bas.bg }

\author{Peter Pflug}
\address{Carl von Ossietzky Universit\"at Oldenburg, Institut f\"ur Mathematik,
Postfach 2503, D-26111 Oldenburg, Germany}
\email{peter.pflug@uni-oldenburg.de}

\thanks{This note was written during the stay of the first named author at the
University of Oldenburg supported by a DAAD grant (November 2009 -
January 2010). The authors are deeply thankful to Pascal J.~Thomas
for pointing out an essential mistake in a former version of this
work. They are grateful to the referee for his comments which
surely improved the presentation of the paper.}



\begin{abstract}
We give an example showing that the Kobayashi-Royden pseudometric for a
pseudoconvex domain is, in general, not the deri\-vative of the Lempert function.
\end{abstract}
\maketitle

Let $\D\subset\C$ be the open unit disc. Fix a domain $D\subset\C^n.$
We recall  the definitions of the Lempert function
$l_D$ and the Kobayashi-Royden pseudometric $\kappa_D$ of $D$:
$$\aligned
l_D(z,w)&=\inf\{|\alpha|:\exists\phi\in\O(\D,D):\phi(0)=z,\phi(\alpha)=w\},\\
\kappa_D(z;X)&=\inf\{|\alpha|:\exists\phi\in\O(\D,D):\phi(0)=z,\alpha
\phi'(0)=X\},
\endaligned$$
where $z,w\in D$ and $X\in\C^n.$

By a result of M.-Y.~Pang (see \cite{Pang}), the Kobayashi-Royden
metric is the "derivative" of the Lempert function for taut
domains in $\C^n$ (such domains are pseudoconvex). More precisely,
one can show that if $D\subset\C^n$ is a taut domain (i.e.~$\O(\Bbb D,D)$ is a normal family),
then
$$\kappa_D(z;X)=\lim_{\C^\ast\ni t\rightarrow 0,z'\to z,X'\to X}\frac{l_D(z',z'+tX')}{|t|}$$
($\C^\ast:=\C\setminus \{0\}$).
For a more general result see \cite{Nik-Pfl}. There it is
also proved that
\begin{equation}\label{1}
\kappa_D(z;X)\ge\CD l_D(z;X):=\limsup_{\C^\ast\ni t\rightarrow 0,z'\to z,X'\to X}\frac{l_D(z',z'+tX')}{|t|}
\end{equation}
for any domain $D\subset\C^n.$ Note that there is a bounded pseudoconvex domain $D\subset\C^2$ containing the
origin such that
$\lim_{\C^\ast\ni t\to 0}\frac{l_D(0,tX)}{|t|}$ does not exist (cf. \cite[Example 4.2.10]{Zwo}), where $X:=(1,1)$. 
Therefore, taking $\limsup$ in the previous definition is needed.

The aim of this note is to show that, in general, the inequality
\begin{equation}\label{2}
\kappa_D(z;X)\ge\wdtl \CD l_D(z;X):=\limsup_{\C^\ast\ni t\rightarrow 0}\frac{l_D(z,z+tX)}{|t|}.
\end{equation}
is a strict one.
\smallskip

Denote by $\M_3$ the set of all $3\times 3$ complex matrices and
by $\Om_3\subset\C^9$ the \textit{spectral unit ball}, i.e.~the
set of all matrices from $\M_3$ with all their eigenvalues in
$\D.$

For a matrix $C\in\M_3$ with eigenvalues $\lambda_1,\lambda_2,\lambda_3,$
we define $$\sigma(C)=(\lambda_1+\lambda_2+\lambda_3,
\lambda_1\lambda_2+\lambda_2\lambda_3+\lambda_3\lambda_1,\lambda_1\lambda_2\lambda_3)\in\C^3.$$

Recall that  $\G_3:=\sigma(\Om_3)$ is the so-called
\textit{symmetrized three-disc}. We will need that $\G_3$ is a
taut domain (even hyperconvex, see e.g.~\cite{Edi-Zwo}).

Put
$$A:=\left(\begin{array}{ccc}0&0&0\\0&0&1\\0&0&0\end{array}\right)\text{\ \ and\ \ }
B_t:=\left(\begin{array}{ccc}1&0&0\\0&\omega&0\\0&3t&\omega^2\end{array}\right),\;t\in\C,$$
where $\omega:=e^{2\pi i/3}$. Set $B:=B_0.$ Now we can formulate our result.

\begin{proposition}\label{Prop-Ex} \rm{(a)} $\kappa_{\Om_3}(A;B)>0=\wdtl\CD l_{\Om_3}(A;B).$

\rm{(b)} Moreover, let $(t_j)_j\subset\C^\ast,$ $(C_j)_j\subset\M_3$ ($C_j=(c_{k,l}^j)$) be such that $t_j\to 0,$ $C_j\to B,$ and
$\liminf_{j\to\infty}|c_{3,2}^j/t_j-3|>0.$ Then
$$\lim_{j\to\infty}\frac{l_{\Om_3}(A,A+t_jC_j)}{|t_j|}=0.$$
\end{proposition}

Since $\kappa_D$ and $l_D$ have the product property, it follows that, in general,
the inequality (\ref{2}) is strict for pseudoconvex domains in $\C^n$ for any $n\ge 9.$
In fact, the proof below shows that
$\wdtl\CD l_{\wdtl\Om_3}(A;B)=0,$
where $\wdtl\Om_3$ is the set of all traceless matrices in $\Om_3.$ So the inequality in (\ref{2}) is strict
for the pseudoconvex domain $\wdtl\Om_3\subset\C^8.$ This remark is due to Pascal J.~Thomas.
\smallskip

\textit{Problem.} It would be interesting to find such examples also in lower dimensions, as well as to
see if, in general, the inequality (\ref{1}) is strict (as it is conjectured in \cite{Nik-Pfl}).
\smallskip

Note that the condition in Proposition \ref{Prop-Ex}(b) implies that the matrices $A+t_jC_j$ are cyclic
for large $j$ (which is, in fact, what we need in the proof). We point out that without the $\liminf$-condition 
the claim in Proposition \ref{Prop-Ex}(b) might not hold. Indeed, we have the following result.

\begin{example}\label{counter} $\ds 1=\kappa_{\wdtl\Om_3}(A;B)=\lim_{\C^\ast\ni t\rightarrow 0}\frac{l_{\Om_3}(A,A+tB_t)}{|t|}.$
In particular, $$1=\kappa_{\wdtl\Om_3}(A;B)=\kappa_{\Om_3}(A;B)=\CD l_{\wdtl\Om_3}(A;B)=\CD l_{\Om_3}(A;B).$$
\end{example}

Before we prove Proposition \ref{Prop-Ex} we need the
following preparation which is based on \cite[Proposition
4.1]{Tho-Tra}.  Recall that $M\in\M_3$ is said to be
\textit{cyclic} if $M$ has a cyclic vector, i.e.
$\span(v,Mv,M^2v)=\C^3$ for some $v\in\C^3$; for many equivalent
properties see e.g.~\cite{Hor-Joh}.

\begin{lemma}\label{lift-1} Let $M\in\Om_3$ be cyclic and $\phi\in\O(\D,\G_3)$  be such that
$\phi(0)=0$ and $\phi(\alpha)=\sigma(M)$ ($\alpha\in\D$). Then there
exists a $\psi\in\O(\D,\Om_3)$ satisfying $\psi(0)=A,$ $\psi(\alpha)=M$ and $\phi=\sigma\circ\psi$
if and only if $\phi_3'(0)=0$.

In particular,
$$
l_{\Om_3}(A,M)=
\inf\{|\alpha|:\exists\phi\in\O(\D,\G_3):\phi(0)=0,\phi(\alpha)=\sigma(M),\phi'_3(0)=0\}$$
and (since $\G_3$ is a taut domain) there is an extremal disc for $l_{\Om_3}(A,M)$.
\end{lemma}

For the convenience of the Reader we give the proof.

\begin{proof} If such a $\psi$ exists, then straightforward calculations show that
$\phi_3'(0)=(\sigma_3\circ\psi)'(0)=0.$

Conversely, assume that $\phi_3'(0)=0$. Put
$$\tilde\psi(\zeta):=\left(\begin{array}{ccc}0&\zeta&0\\0&0&1\\
\phi_3(\zeta)/\zeta&-\phi_2(\zeta)&\phi_1(\zeta)\end{array}\right),\quad\zeta\in\D.$$
Then $\tilde\psi(0)=A$ and $\phi=\sigma\circ\tilde\psi.$ Note also
that $(0,0,1)$ is a cyclic vector for $\tilde\psi(\zeta)$ if
$\zeta\neq 0.$ So $\tilde\psi(\alpha)$ is a cyclic matrix with the
same spectrum as the cyclic matrix $M$ and hence they are
conjugate (cf.~\cite{Hor-Joh}). It remains to write $M$ in the
form $M=e^{-S}\tilde\psi(\alpha)e^S$ for some $S\in\M_3$ and to
set $\psi(\zeta)=e^{-\zeta S/\alpha}\tilde\psi(\zeta)e^{\zeta
S/\alpha}.$
\end{proof}

Now we are able to present the proof of Proposition \ref{Prop-Ex}.

\begin{proof}[Proof of Proposition \ref{Prop-Ex}] In virtue of Example \ref{counter}, we have only
to verify that
$$
\lim_{j\to\infty}\frac{l_{\Om_3}(A,A+t_jC_j)}{|t_j|}=0
$$
under the above condition on the $c_{3,2}^j$.

\textit{STEP 1.} First we prove that the $\liminf$-condition implies that $A+t_jC_j$ are cyclic matrices for sufficiently large $j$'s. Assume that all (otherwise take an appropriate subsequence) $A+t_jC_j$ are non cyclic matrices. Therefore, their minimal
polynomials are of degree less than $3$ (cf.~\cite{Hor-Joh}).
So their degrees are equal to $2$ for sufficiently large $j.$ Then
$$
(A+t_jC_j)^2+x_j(A+t_jC_j)+y_jE=0,\quad j\in\N,
$$
where $x_j, y_j\in\C$, and $E$ denotes the unit matrix in $\M_3$. So we get $9$ equations;
each of them is denoted by $E_{k,\ell}^j$, where the indices $k$ and $\ell$ denote the row and the column,
respectively. Looking at equation $E_{2,3}^j$ we get $x_j/t_j\too 1$.
Putting this into equation $E_{1,1}^j$ leads to $y_j/t_j^2\too -2$. Finally, equation $E_{2,2}^j$ implies that
$c_{3,2}^j/t_j\too 2-\omega-\omega^2=3$; a contradiction.

\textit{STEP 2.} By step 1 we know that  all
matrices $A+t_jC_j$  are cyclic and belong to $\Om_3$ if  $j\geq j_0$. Calculations show that
$$
\sigma(A+t_jC_j)=(t_jf_1(C_j),t_jf_2(C_j),t_j^2f_3(C_j))=:(a_j,b_j,c_j),
$$
with $f_1(C_j)\too 0$, $f_2(C_j)\too 0$, and $f_3(C_j)\too 0$.

Put
$$
\phi_j(\zeta):=(\zeta a_j/r_j, \zeta b_j/r_j,\zeta^2 c_j/r_j^2), \quad\zeta\in\D,
$$
where $r_j:=\max\{3|a_j|,3|b_j|,\sqrt{3|c_j|}\}$. Then $\phi_j\in\O(\D,\G_3)$ with
$\phi_j(0)=0$, $\phi_{j,3}'(0)=0$, and $\phi_j(r_j)=\sigma(A+t_jC_j)$.
Hence, by Lemma \ref{lift-1},
$$
l_{\Om_3}(A,A+t_jC_j)/|t_j|\leq r_j/|t_j|\to 0.
$$
Hence the proof is finished.
\end{proof}

Finally we present the proof of the example.

\begin{proof}[Proof of Example \ref{counter}] Since $A+\zeta B\in\wdtl\Om_3$ for any $\zeta\in\D,$
it follows that $\kappa_{\wdtl\Om_3}(A;B)\le 1.$

By (1), it remains to show that $\liminf_{\C^\ast\ni
t\rightarrow 0}\frac{l_{\Om_3}(A,A+tB_t)}{|t|}\ge 1.$

Note that $A+tB_t$ is similar to the matrix
$D_t=\mbox{diag}(t,t,-2t)$ and hence $l_{\Om_3}(A,A+tB_t)
=l_{\Om_3}(A,D_t)$ (use the same argument as the one at the end of
the proof of Lemma \ref{lift-1}).

Let $(t_j)_j\subset\C^\ast$, $t_j\to 0$, such that
$l_{\Om_3}(A,D_{t_j})/|t_j|\to c$.

Now choose
$\psi_j=(\psi_{j,(kl)})_{k,l=1,2,3}\in\O(\D,\Om_3)$  such that
$\psi_j(0)=A$, $\psi(\alpha_j)=D_{t_j}$, and
$\alpha_j/t_j=|\alpha_j|/|t_j|\to c$. Setting
$$
\phi_j=(\phi_{j,1},\phi_{j,2},\phi_{j,3}):=\sigma\circ\psi_j,
$$
we have the following equations:
$$\phi_{j,1}=\trace \psi_j,\ \ \phi_{j,3}=\det \psi_j,$$
$$\begin{aligned}
\phi_{j,2}&=\psi_{j,(11)}\psi_{j,(22)}+\psi_{j,(11)}\psi_{j,(33)}+\psi_{j,(22)}\psi_{j,(33)}\\
&-\psi_{j,(12)}\psi_{j,(21)}-\psi_{j,(13)}\psi_{j,(31)}-\psi_{j,(23)}\psi_{j,(32)}.
\end{aligned}
$$
Then
straightforward  calculations show that $\phi'_{j,3}(0)=0$ and
$$
\phi_{j,3}'(\alpha_j)-t_j\phi_{j,2}'(\alpha_j)+t_j^2\phi_{j,1}'(\alpha_j)=0.
$$
Writing
$$
\phi_j(\zeta)=(\zeta\theta_{j,1}(\zeta),\zeta\theta_{j,2}(\zeta),
\zeta^2\theta_{j,3}(\zeta)),
$$
the last condition becomes
\begin{equation}\label{prel}
t_j^3=\alpha_j^2(\alpha_j\theta_{j,3}'(\alpha_j)-t_j\theta_{j,2}'(\alpha_j)+t_j^2\theta_{j,1}'(\alpha_j))
\end{equation}
(use that $\theta_{j,1}(\alpha_j)=0,$
$\theta_{j,2}(\alpha_j)=-3t_j^2/\alpha_j$ and
$\theta_{j,3}(\alpha_j)=-2t_j^3/\alpha_j^2$). Since $\G_3$ is a
taut domain, passing to a subsequence, we may assume that
$\phi_j\to\phi=(\zeta\rho_1,\zeta^2\rho_2,\zeta^3\rho_3)\in\O(\D,\G_3)$
and $\rho_1(0)=0.$ Then the equation (\ref{prel}) implies that
$$\rho_3(0)=k^3+k\rho_2(0),$$
where $k:=1/c.$

It follows by \cite[Proposition 1]{Nik} (see also
\cite[Proposition 16]{Edi-Zwo}) that
$h_{\G_3}(z):=\max\{|\lambda|:\lambda^3-z_1\lambda^2+z_2\lambda-z_3=0\}$
is a (logarithmically) plurisubharmonic function with  $\G_3=\{z\in\C^3:h_{\G_3}(z)<1\}$. In fact, $h_{G_3}$ is
the Minkowski function of the $(1,2,3)$-balanced domain $\G_3$. Since
$$|\zeta| h_{\G_3}(\rho_1(\zeta),\rho_2(\zeta),\rho_3(\zeta))=h_{\G_3}(\phi(\zeta))<1,\quad \zeta\in\D,$$
the maximum principle for plurisubharmonic functions implies that
$h_{\G_3}(\rho_1,\rho_2,\rho_3)\leq 1$ on $\D$. In particular,
$h_{\G_3}(\rho_1(0),\rho_2(0),\rho_3(0))\leq 1$. Therefore, all
zeros of the polynomial
$P(\lambda):=\lambda^3-\rho_1(0)\lambda^2+\rho_2(0)\lambda-\rho_3(0)$,
lie in $\overline\D$. Note that
$P(\lambda)=(\lambda-k)(\lambda^2+k\lambda+k^2+\rho_2(0))$; hence
$c\geq 1$.
\end{proof}

\end{document}